\documentclass[11pt]{amsart}
\usepackage{amsmath,amsthm}
\usepackage{amssymb}
\usepackage{enumerate}
\usepackage{hyperref}
\usepackage{color}

\usepackage{graphicx}

\usepackage{pgfplots}
\usepackage{mathrsfs}
\usetikzlibrary{arrows}
\usetikzlibrary[patterns]

\usepackage{enumitem}

\newtheorem{theorem}{Theorem}[section]
\newtheorem{proposition}[theorem]{Proposition}
\newtheorem{lemma}[theorem]{Lemma}

\newtheorem{example}[theorem]{Example}
\newtheorem{remark}[theorem]{Remark}

\newcommand{\C}{\mathbb{C}}
\newcommand{\SC}{\mathbb{SC}}
\newcommand{\Q}{\mathbb{Q}}
\newcommand{\N}{\mathbb{N}}
\newcommand{\R}{\mathbb{R}}
\newcommand{\K}{\mathcal{K}}
\newcommand{\F}{\mathcal{F}}
\newcommand{\G}{\mathcal{G}}
\newcommand{\I}{\mathbb{I}}
\newcommand{\notion}[1]{\emph{#1}}
\newcommand{\id}{\operatorname{id}}

\begin{document}
	

	
	\title{
		Split square and split carpet as 
		examples of non-metrizable IFS attractors
	}
	
	\author{Krzysztof Le\'{s}niak}
	\address{Faculty of Mathematics and Computer Science, 
		Nicolaus Copernicus University in Toru\'{n}\\
		Chopina 12/18, 87-100 Toru\'{n}, Poland\\
		ORCID 0000-0001-5992-488X\quad\ E-mail: much@mat.umk.pl}
	
	\author{Magdalena Nowak}
	\address{Mathematics Department, Jan Kochanowski University in Kielce\\ 
		Uniwersytecka 7, 25-406 Kielce, Poland\\
		ORCID 0000-0003-1915-0001\quad\ E-mail: magdalena.nowak805@gmail.com }
	
	\maketitle

	\begin{abstract}
		We define two non-metrizable compact spaces 
		and show that they are attractors 
		of iterated function systems. 
		Both of them, the split square and the split carpet, 
		are constructed using the product of split intervals. 
	\end{abstract}

	\section{Introduction}
	
	The theory of iterated function systems (IFS) 
	defined on general topological spaces
	is fairly well developed,
	cf. \cite{BKNNS, Barrientos2017, Barrientos2020,
		FitzsimmonsKunze-June2020, Kameyama, Kieninger,
		McGehee, MiculescuMihail, Tetenov2017}. 
	In that regard, various concepts of attractors 
	have been introduced. However, 
	the examples of IFSs that 
	are defined on non-metrizable spaces 
	and admit an attractor 
	seem to be very scarce, e.g.,
	\cite[Example 6]{BLesR2016chaos}, 
	\cite[Example 3.8]{LSS-Manual}.
	In principle, to construct IFSs with non-metrizable 
	attractors, one could use point transitive dynamical 
	systems (in particular, minimal dynamical systems);
	see \cite[Example 2]{BLesR2018fractals} for 
	a specific construction which works in general. 
	However, no efficient characterization of spaces 
	admitting transitive (minimal) systems is known, 
	even in the case of metrizable spaces, and 
	the advance of the subject is quite scattered,
	see for instance \cite{BoCiFo2021}.
	Moreover, the amount of standard examples of 
	non-metrizable spaces that are compact and separable 
	(a necessary condition for an attractor of an IFS 
	\cite[Proposition 5]{BLesR2016chaos}) is rather small, 
	e.g., \cite{PiBase}.

	The aim of this work is to 
	provide two examples of IFSs 
	with non-metrizable attractors. 
	They are inspired by the IFS
	in \cite[Example 6]{BLesR2016chaos},
	which is acting on the split interval 	
	(also called double arrow space \cite{PiBase} 
	or two arrows space \cite{Engelking}).	

	In section 2 we present a construction 
	and properties of the first example, the split square. 
	We show that it is the attractor of 
	an iterated function system constructed in section 3. 
	The second example, called split carpet, is described 
	in the last section.
	
	\section{Split square}
	
	A \notion{split square} is the set 
	\begin{align*}
		\Q = [0,1)\times(0,1]\times\{1\}
		& \quad\cup\; (0,1]\times(0,1]\times\{2\}
		\\
		\cup\; (0,1]\times[0,1)\times\{3\}
		& \quad\cup\; [0,1)\times[0,1)\times\{4\}.
	\end{align*}
	We endow it with the topology 
	generated by the basis of sets
	\begin{align}
		\label{eq:base}
		Q(a,b;c,d) := [a,b)\times(c,d]\times\{1\}
		&\quad \cup\; (a,b]\times(c,d]\times\{2\}
		\\
		\notag
		\cup\; (a,b]\times[c,d)\times\{3\}
		& \quad\cup\; [a,b)\times[c,d)\times\{4\},
	\end{align}
	\(0\leq a<b\leq 1\), \(0\leq c < d\leq 1\);
	see Fig. \ref{fig:splitnbd}.
	It is customary to set
	\(Q(a,b;c,d) := \emptyset\) when 
	\(a\geq b\) or \(c\geq d\).
	One can easily see that \eqref{eq:base}
	forms a basis of some topology.
	Indeed, 
	\begin{align*}
		Q(a,b; c,d) \cap Q(a',b'; c',d') =
		Q(a\vee a', b\wedge b'; c\vee c', d\wedge d'),
	\end{align*}
	where 
	\(r\vee s:=\max(r,s)\), 
	\(r\wedge s:=\min(r,s)\).
	In particular, 
	the above intersection is empty 
	when \(a\vee a' \geq b\wedge b'\)
	or \(c\vee c' \geq d\wedge d'\).
	
	The basic open sets \eqref{eq:base} 
	are clopen thanks to the equality
	\begin{align*}
		\Q\setminus Q(a,b; c,d) = 
		Q(0,1; 0,c) \cup
		Q(0,1; d,1) \cup
		Q(0,a; 0,1) \cup
		Q(b,1; 0,1).
	\end{align*}
	Therefore, \(\Q\) is zero-dimensional; 
	see Theorem \ref{th:propertiesofQ} 
	for more properties of \(\Q\).
	
	It should be remarked that 
	a space nearly identical 
	to the split square was 
	considered by T. Banakh in \cite{Banakh} 
	under the same name. 
	Banakh's split square differs from \(\Q\) 
	by ``edges'' as it is a product of two 
	split intervals with uncut ``isolated corners'';
	see Proposition \ref{prop:QisI2}.
	Note also that each slice 
	\(\mathbb{Q}\cap\big( [0,1]^2\times\{t\}\big)\),
	\(t=1,2,3,4\), bears the topology induced 
	from the Sorgenfrey plane (a product of two 
	Sorgenfrey lines); see Fig. \ref{fig:splitnbd}.
	
	\begin{figure}
		\begin{center}
			\begin{tikzpicture}[line cap=round,line join=round,>=triangle 45,x=1.0cm,y=1.0cm, scale=0.8]
				\clip(0.8,-0.6) rectangle (3.73,5.37);
				\fill[line width=2.pt,color=black,fill=black,pattern=dots,pattern color=black] (2.,5.0) -- (1.6,4.0) -- (2.6,4.0) -- (3.,5.) -- cycle;
				\fill[line width=2.pt,color=black,fill=black,pattern=dots,pattern color=black] (2.0,3.6) -- (1.6,2.6) -- (2.6,2.6) -- (3.0,3.6) -- cycle;
				\fill[line width=2.pt,color=black,fill=black,pattern=dots,pattern color=black] (2.0,2.2) -- (3.0,2.2) -- (2.6,1.2) -- (1.6,1.2) -- cycle;
				\fill[line width=2.pt,fill=black,pattern=dots,pattern color=black] (2.0,0.8) -- (3.0,0.8) -- (2.6,-0.2) -- (1.6,-0.2) -- cycle;
				\draw [line width=2.pt,color=black] (2.,5.)-- (1.6,4.0);
				\draw [line width=2.pt,color=black] (1.6,4.0)-- (2.6,4.0);
				\draw [line width=2.pt,color=black] (1.6,2.6)-- (2.6,2.6);
				\draw [line width=2.pt,color=black] (2.6,2.6)-- (3.0,3.6);
				\draw [line width=2.pt,color=black] (2.0,2.2)-- (3.0,2.2);
				\draw [line width=2.pt,color=black] (3.0,2.2)-- (2.6,1.2);
				\draw [line width=2.pt] (2.0,0.8)-- (3.0,0.8);
				\draw [line width=2.pt] (1.6,-0.2)-- (2.0,0.8);
				\begin{scriptsize}
					\draw [color=black] (2.0,5.1) circle (2.5pt);
					\draw [fill=black] (1.6,4.0) circle (2.5pt);
					\draw [color=black] (2.7,4.0) circle (2.5pt);
					\draw [color=black] (1.5,2.6) circle (2.5pt);
					\draw [fill=black] (2.6,2.6) circle (2.5pt);
					\draw [color=black] (3.0,3.7) circle (2.5pt);
					
					\draw [color=black] (1.9,2.2) circle (2.5pt);
					\draw [fill=black] (3.0,2.2) circle (2.5pt);
					\draw [color=black] (2.6,1.1) circle (2.5pt);
					
					\draw [fill=black] (2.0,0.8) circle (2.5pt);
					\draw [color=black] (3.1,0.8) circle (2.5pt);
					\draw [color=black] (1.6,-0.3) circle (2.5pt);
				\end{scriptsize}
			\end{tikzpicture}
		\end{center}
		\caption{The shape of \(\Q\) and basic neighbourhoods.}
		\label{fig:splitnbd}
	\end{figure}
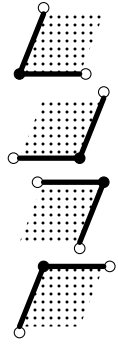

	What is crucial about \(\Q\) is that 
	it is the topological product of 
	two split intervals. Let 
	\(\I :=\Big( (0,1]\times\{0\} \Big)\;\cup\;\Big( [0,1)\times\{1\}\Big)\) 
	be the \notion{split interval}. 
	The topology on \(\I\) is generated 
	by a basis of sets of two forms 
	\begin{align*}
		I_0(a,b) := 
		\Big((a,b]\times\{0\} \Big)
		& \;\cup\; \Big((a,b)\times\{1\}\Big),
		\\
		I_1(a,b) := 
		\Big((a,b)\times\{0\}\Big)
		& \;\cup\; \Big([a,b)\times\{1\}\Big),
	\end{align*} 
	cf.
	\cite[Exercise 3.10.C chap.3 p. 212]{Engelking}.
	We find more convenient to 
	work with the basis formed by
	the scaled split intervals
	\begin{align*} 
		I(a,b) :=\Big( (a,b]\times\{0\} \Big)
		\;\cup\; \Big([a,b)\times\{1\}\Big),
	\end{align*}
	where \(0\leq a<b\leq 1\).
	To see the equivalence of the two bases,
	one just observes that
	$$I(a,b) = I_0(a,b)\cup I_1(a,b), I_0(a,b) = \bigcup_{n\in\N} I(a+(b-a)/n,b),
	I_1(a,b) = \bigcup_{n\in\N} I(a,b-(b-a)/n). $$

	\begin{proposition}\label{prop:QisI2}
		The split square $\Q$ is homeomorphic to $\I^2$, 
		the product of two split intervals.
	\end{proposition}
	
	\begin{proof}
		Let us consider the product
		\begin{align*}	
			\I^2 
			&= (0,1]\times\{0\}\times(0,1]\times\{0\}
			\\
			\mbox{} 
			&\cup\; (0,1]\times\{0\}\times[0,1)\times\{1\}
			\\
			\mbox{}
			&\cup\; [0,1)\times\{1\}\times(0,1]\times\{0\}
			\\
			\mbox{}
			&\cup\; [0,1)\times\{1\}\times[0,1)\times\{1\}.
		\end{align*}
		Define \(h\colon \Q\to \I^2\) by 
		\begin{align}
			\label{eq:homeo}
			h(x,y,t) :=
			\begin{cases}
				(x,1,y,0) & \text{ when }t=1,\\
				(x,0,y,0) & \text{ when }t=2,\\
				(x,0,y,1) & \text{ when }t=3,\\
				(x,1,y,1) & \text{ when }t=4,
			\end{cases}
		\end{align}
		for \((x,y,t)\in\Q\).
		The map \(h\) is a homeomorphism.
		Indeed, \(h\) is bijective, and 
		\begin{align}
			\label{eq:hQIxI}
			h(Q(a,b;c,d)) = I(a,b)\times I(c,d)
		\end{align}
		for \(a<b, c<d\) in \([0,1]\). 
	\end{proof}

	Basic properties of \(\Q\) are gathered below.
	\begin{theorem}\label{th:propertiesofQ}
		The split square \(\Q\) is a non-metrizable
		zero-dimensional compact separable Hausdorff space.
	\end{theorem}
	\begin{proof}
		Follows from Proposition \ref{prop:QisI2}
		and properties of the split interval,
		cf. \cite{Engelking}. 
	\end{proof}

	We finish this section with 
	a lemma providing a wide class of 
	continuous functions on \(\I\) and \(\Q\).
	
	The product of two maps 
	\(f\colon X\to X\) and \(g\colon Y\to Y\)
	is denoted by 
	\(f\times g: X\times Y\to X\times Y\).
	Recall that 
	\(f\times g\,(x,y) :=(f(x),g(y))\)
	for \((x,y)\in X\times Y\). 
	The restriction of \(f\times g\) 
	to a subset of \(X\times Y\)
	will be also denoted by \(f\times g\). 
	If \(f\) and \(g\) are continuous and 
	\(X\times Y\) bears the product topology,
	then \(f\times g\) is continuous.
	The identity map will be denoted by \(\id\).
	
	\begin{lemma}\label{lem:continuity}
		Let \(g_1,g_2\colon [0,1] \to [0,1]\) 
		be increasing, continuous maps in the Euclidean sense. 
		\begin{enumerate}
			\item[(a)]
			The map \(g_1\times \id\) defined on \(\I\)
			is continuous in the topology of split interval.
			\item[(b)]
			The map \(g_1\times g_2\times \id\) 
			defined on \(\Q\) is continuous 
			in the topology of split square.
		\end{enumerate}
	\end{lemma}
	\begin{proof}
		Note that the preimages via 
		increasing maps \(g_i\) 
		preserve shapes of intervals.
		Therefore, 
		\begin{align*}
			(g_1\times\id)^{-1}(I(a,b))
			= I(a',b'),
			\\
			(g_1\times g_2\times\id)^{-1}(Q(a,b;c,d)) 
			= Q(a',b';c',d'), 
		\end{align*}
		for \(a<b\) and \(c<d\) in \([0,1]\),
		where 
		\(a' = g_1^{-1}(a)\), 
		\(b' = g_1^{-1}(b)\),
		\(c' = g_2^{-1}(c)\),
		\(d' = g_2^{-1}(d)\).
	\end{proof}
	
	\begin{remark}\label{rem:homeo}
		The base \(I(a,b) \times I(c,d)\), 
		\(a<b\), \(c<d\), in \(\I^2\) 
		corresponds to the base 
		\eqref{eq:base} in \(\Q\) according 
		to the formula \eqref{eq:hQIxI}. 
		Similarly, in the context of 
		Lemma \ref{lem:continuity}, 
		the product of two maps
		\(g_1\times \id\) and \(g_2\times \id\)
		on \(\I\) stays in the natural correspondence
		with the map 
		\(g_1\times g_2\times \id\) on \(\Q\). 
		Namely,
		\begin{align*}
			h \circ (g_1\times g_2\times \id)
			= (g_1 \times \id) \times (g_2\times \id), 
		\end{align*}
		where \(h\) is given by \eqref{eq:homeo}.
	\end{remark}

	\section{IFSs on the split square}
	
	In this work by an 
	\notion{iterated function system} (shortly IFS)
	we understand  a finite collection 
	of continuous maps \(f_{i}:X\to X\), \(i\in I\), 
	acting on a Hausdorff topological space \(X\),
	e.g., \cite{BLesR2016chaos, Barrientos2017}.
	We write \(\F = (X; f_i:i\in I)\).
	Let \(\K(X)\) be a hyperspace of 
	nonempty compact subsets of \(X\).
	We endow \(\K(X)\) with 
	the \notion{Vietoris topology} whose
	subbasic open sets are of the following two forms:
	\begin{align*}
		V^{+} := \{S\in\K(X): S\subseteq V\}=\K(V),
		\\
		V^{-} := \{S\in\K(X): S\cap V\neq\emptyset\},
	\end{align*}
	where \(V\) runs over all open sets in \(X\),
	e.g., \cite{HuPa}. 
	Note that a sequence of compact sets 
	is convergent in the Vietoris sense if
	and only if it is convergent in the 
	\notion{upper Vietoris topology}, 
	which is generated by the sets \(V^{+}\),
	and the \notion{lower Vietoris topology},
	which is generated by the sets \(V^{-}\),
	e.g., \cite[Sec.10]{LowerUpper}.
	We associate with an IFS \(\F\) 
	the \notion{Hutchinson operator}
	\(F:\K(X)\to\K(X)\). 
	We will write \(F^n\) for the \(n\)-fold 
	composition of \(F\) and, for simplicity,
	\(F^n(x)\) instead of \(F^n(\{x\})\).
	
	A set \(A\in\K(X)\) is a \notion{strict attractor}
	of the IFS \(\F\) provided 
	there exists an open \(U\supseteq A\) such that
	for all \(K\in\K(U)\),
	\begin{align*}
		F^n(K)
		\underset{n\to\infty}{\longrightarrow} A
	\end{align*}
	in the Vietoris sense. The maximal open set $U$
	satisfying the above property is called 
	a~\notion{basin} of the attractor $A$. 
	If one can take \(U=X\), then \(A\) 
	is said to have a \notion{full basin}. 
	In particular, \(F(A)=A\) because \(F\) is 
	continuous in the Vietoris topology,
	cf. \cite{BLesContinuity}.
	It should be noted that a strict attractor \(A\) 
	is necessarily a separable space,
	e.g., \cite[Proposition 5]{BLesR2016chaos}.
	
	Further on, we use letters
	\(\F\), \(\G\) for the IFSs 
	and \(F\), \(G\) for the 
	corresponding Hutchinson operators.
	
	For example, the unit interval \([0,1]\) 
	is a strict attractor of the IFS 
	\((\R; g_i : i=0,1)\) 
	and its restriction
	\begin{align}
		\label{eq:IFSonI}
		\G := ([0,1]; g_i : i=0,1),
	\end{align} 
	where 
	\begin{align}
		\label{eq:IFSonImaps}
		g_i(x):= \frac{x+i}{2}.
	\end{align} 
	Similarly, the square \([0,1]^2\) 
	is a strict attractor of the IFS 
	\((\R^2; g_i\times g_j : i,j\in\{0,1\})\),
	with \(g_i\) given by \eqref{eq:IFSonImaps}. 
	
	It was shown in \cite[Example 6]{BLesR2016chaos} 
	that the split interval \(\I\) is 
	a strict attractor of a suitably defined 
	IFS on \(\I\). We modify this construction
	below by omitting one of the maps 
	in the original IFS. 
	
	\begin{example}\label{ex:IFSonII}
		Let \(\F = (\I; f_i: i=0,1)\), where
		\(f_i(x,t) := \left(\frac{x+i}{2}, t\right)\)
		for \((x,t)\in \I\). 
		Then \(\I\) is a strict attractor 
		of \(\F\) with full basin.
		
		Fix \((x,t)\in\I\). 
		The upper Vietoris convergence 
		\(F^n(x,t)\to\I\) is obvious.
		To show the lower Vietoris convergence,
		let \(V\subseteq \I\) be 
		an arbitrary nonempty open set;
		that is \(V\cap\I\neq\emptyset\).
		Then \(V\supseteq I(a,b)\)
		for some \(0\leq a<b\leq 1\).
		Denote \(U := (a,b)\).
		Since \([0,1]\) is the attractor 
		of the IFS \(\G\) defined in 
		\eqref{eq:IFSonI}, \eqref{eq:IFSonImaps}, 
		there exists \(n_0\in\N\) such that
		\(U\cap G^n(x)\neq\emptyset\)
		for \(n\geq n_0\). Therefore 
		\begin{align*}
			V\cap F^n(x,t) \supseteq 
			U\times\{t\} \cap G^n(x)\times \{t\} 
			\neq \emptyset
		\end{align*}
		for \(n\geq n_0\), 
		due to the relations: 
		\(F^n(x,t)=G^n(x)\times\{t\}\),
		and \(V\supseteq U\times\{t\}\).
		To finish the reasoning,
		one needs to ensure the Vietoris convergence 
		\(F^n(K)\to\I\) for every \(K\in\K(\I)\),
		not only for \(K=\{(x,t)\}\).
		This is immediate by the squeezing argument,
		since
		\begin{align*}
			F^n(x,t)\subseteq F^n(K)\subseteq \I,
		\end{align*}
		where \((x,t)\) is a point 
		picked anyhow from the set \(K\).
	\end{example}
	
	Let \(\F_1\) be an IFS on \(X_1\) 
	and let \(\F_2\) be an IFS on \(X_2\). 
	We define the \notion{product IFS} as follows
	\(\F_1\times \F_2 := 
	(X_1\times X_2; f_{1}\times f_{2} :
	f_1\in\F_1, f_2\in\F_2)\).
	This can be generalized to several factors.
	The projections 
	\(\pi_k: X_1\times X_2 \to X_k\), \(k=1,2\),
	are defined as \(\pi_k(x_1,x_2) := x_k\).
	
	It turns out that a finite product 
	of strict attractors is again 
	a strict attractor. 
	We formulate a suitable
	lemma for two factors.
	
	\begin{lemma}\label{lem:prodIFS}
		Let \(\F_1\) and \(\F_2\) be two IFSs 
		that are defined on \(X_1\) and \(X_2\),
		respectively, and which have strict attractors 
		\(A_1\) and \(A_2\) with basins \(U_1\) and \(U_2\). 
		Then \(A_1\times A_2\) is a strict attractor
		of the product system \(\F_1\times \F_2\)
		with a basin that contains \(U_1\times U_2\).
	\end{lemma}
	\begin{proof}
		For every \(K\in\K(U_1\times U_2)\), 
		\((x_1,x_2)\in K\) and \(n\in\N\) 
		we have 
		\begin{align*}
			F_1^n(x_1) \times F_2^n(x_2)
			=
			(F_1\times F_2)^n\,(x_1,x_2) 
			\subseteq 
			(F_1\times F_2)^n(K)
			\\ 
			\subseteq (F_1\times F_2)^n(\pi_1(K)\times\pi_2(K))
			\\	
			= F_1^n(\pi_1(K)) \times F_2^n(\pi_2(K)),
		\end{align*}
		Since \(x_k\in \pi_k(K)\subseteq U_k\), \(k=1,2\),
		by applying to the above inclusion
		the standard properties of Vietoris limits 
		(cf. \cite{HuPa,Illanes}) 
		and the squeezing argument, 
		we can finish the proof.
	\end{proof}

	The above considerations allow us to formulate
	
	\begin{theorem}
		The split square \(\Q\) is 
		a strict attractor of the IFS 
		\((\Q; f_{ij}: i,j\in\{0,1\})\), 
		\begin{align*}
			f_{ij}(x,y,t) := 
			\left(\frac{x+i}{2}, \frac{y+j}{2}, t\right),
			\quad (x,y,t)\in\Q.
		\end{align*}
	\end{theorem}
	\begin{proof}
		Clearly \(f_{ij}= g_i\times g_j\times \id\),  
		\(i,j\in\{0,1\}\), where 
		the maps \(g_i\) are defined 
		in \eqref{eq:IFSonImaps}.
		According to Remark \ref{rem:homeo}, 
		our IFS can be identified with 
		the product of two copies of 
		the IFS from Example \ref{ex:IFSonII}.
		The application of Lemma \ref{lem:prodIFS} 
		completes the proof.
	\end{proof}

	\section{Split carpet}
	
	In this section we construct 
	the split carpet \(\SC\), an analogon of 
	the Sierpi\'{n}ski carpet \(\C\) 
	in the split square \(\Q\),
	and show that \(\SC\) 
	is an attractor of an IFS 
	on \(\Q\).
	
	Let us recall that 
	the \notion{Sierpi\'{n}ski carpet} \(\C\) 
	is the intersection 
	\(\C=\bigcap_{k\in\N} C_k\) 
	of the following sequence 
	of nonempty compact sets 
	\begin{align*}
		{C}_{0} := & 
		[0,1]^2, \\
		{C}_{k} := &
		{C}_{k-1} \setminus
		\bigcup_{0\leq l,m < 3^{k-1}}
		\left(\frac{3l+1}{3^k}, \frac{3l+2}{3^k}\right)
		\times
		\left(\frac{3m+1}{3^k}, \frac{3m+2}{3^k}\right),
		\; k\geq 1.
	\end{align*}
	Moreover, as is well known, 
	\(\C\) is the attractor with full basin 
	of the following IFS: 
	\begin{align}\label{eq:IFSforC}
		\mathcal{G} = 
		([0,1]^2; g_{ij}: 
		(i,j)
		\in\{0,1,2\}^2
		\setminus\{(1,1)\}),
	\end{align}
	where
	\(g_{ij}:[0,1]^2\to[0,1]^2\) and
	\begin{align}\label{eq:mapsCarpet}
		g_{ij}(x,y) := \left(\frac{x+i}{3}, 
		\frac{y+j}{3}\right)
		\quad\text{for}\; (x,y)\in[0,1]^2.
	\end{align} 
	
	Define inductively a nested sequence 
	of nonempty compact sets 
	\begin{align*}
		{SC}_{0} := & 
		Q(0,1;0,1), \\
		{SC}_{k} := &
		{SC}_{k-1} \setminus
		\bigcup_{0\leq l,m < 3^{k-1}}
		Q\left( \frac{3l+1}{3^k},
		\frac{3l+2}{3^k};
		\frac{3m+1}{3^k},
		\frac{3m+2}{3^k}\right),
		\; k\geq 1.
	\end{align*}
	The sets \(SC_{k}\) are formed by 
	cutting off from \(\Q\) split squares \(Q\)
	similarly as it is done in the construction
	of the Sierpi\'{n}ski carpet, 
	see Fig. \ref{fig:splitcarpet}.
	The \notion{split carpet} \(\SC\)
	is the intersection 
	\(\SC:= \bigcap_{k\in\N} {SC}_k\).
	
	\begin{figure}
		\begin{center}
			\includegraphics[scale=0.3]{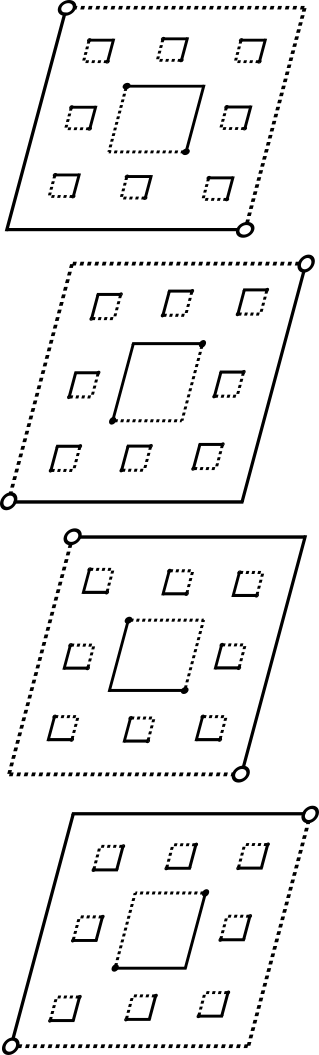}
		\end{center}
		\caption{The second generation 
			\({SC}_{2}\) of the sequence 
			\({SC}_k\) approximating 
			the split carpet \(\SC\).}
		\label{fig:splitcarpet}
	\end{figure}
	
	It should be remarked that \(\SC\)
	is non-metrizable because it contains
	a copy of the split interval:
	\(\SC \cap\big([0,1]\times\{0\} \times \{3,4\}\big)\);
	see Fig. \ref{fig:splitcarpet}.
	
	The following simple lemma
	expresses a strong affinity between 
	\(\SC\) and \(\C\).
	
	\begin{lemma}\label{lem:SC-holes}
		Let \(0\leq a<b\leq 1\), 
		\(0\leq c<d\leq 1\). If 
		\begin{align}\label{eq:SC-intersect}
			Q(a,b;c,d)\cap\SC\neq\emptyset,
		\end{align}
		then 
		\(\big((a,b)\times (c,d)\big)\cap\C\neq\emptyset\).
	\end{lemma}
	\begin{proof}
		Suppose that \(U\cap \C=\emptyset\), 
		where \(U:= (a,b)\times(c,d)\).
		Since \(U\) is a connected set,
		there must be a hole in \(\C\) 
		which contains \(U\). That is 
		\(U\subseteq (a',b')\times(c',d')\)
		for some 
		\(a'= (3l+1)\cdot 3^{-k}\),
		\(b'= (3l+2)\cdot 3^{-k}\),
		\(c'= (3m+1)\cdot 3^{-k}\),
		\(d'= (3m+2)\cdot 3^{-k}\), where
		\(0\leq l,m < 3^{k-1}\) and
		\(k\geq 1\).
		Hence \(Q(a,b; c,d)\) is contained in one
		of the holes cut off from \(\Q\) 
		in the construction of \(\SC\).
		Namely,
		\(Q(a,b;c,d)\subseteq Q(a',b';c',d')\)
		as \(a'\leq a<b\leq b'\),
		\(c'\leq c<d\leq d'\).
		Since \(Q(a',b';c',d')\cap\SC =\emptyset\),
		this contradicts \eqref{eq:SC-intersect}. 
	\end{proof}

	Now we define an IFS on \(\Q\)
	whose attractor is \(\SC\).
	Let \(f_{ij}:\Q\to\Q\),
	\(i,j\in\{0,1,2\}\),
	\begin{align*}
		f_{ij}(x,y,t) := 
		\left(\frac{x+i}{3}, \frac{y+j}{3}, t\right)
		\quad \text{for}\; (x,y,t)\in\Q.
	\end{align*}
	
	\begin{theorem}
		The split carpet \(\SC\) is a strict attractor
		with full basin of the IFS \\
		\(\F=(\Q; f_{ij}: 
		(i,j)\in\{0,1,2\}^{2}\setminus\{(1,1)\})\).
	\end{theorem}
	\begin{proof}
		Fix \(K\in\K(\Q)\). We shall show that
		\begin{align}\label{eq:FnKtoSC}
			F^n(K)\to\SC
		\end{align} 
		in the lower and upper Vietoris topology.
		
		First observe that 
		\(F^n(\Q)={SC}_n\to \SC\)
		in the Vietoris sense, because 
		\(\SC\) is a descending 
		intersection of nonempty and 
		compact sets \({SC}_n\), cf. 
		\cite[Sect.I.4, Theorems 4.4 and 4.6, Exercise 4.16]{Illanes}.
		In particular, we have the 
		upper Vietoris convergence 
		in \eqref{eq:FnKtoSC}.
		
		To verify the lower Vietoris convergence
		in \eqref{eq:FnKtoSC}, 
		by the squeezing argument, 
		it is enough to pick anyhow \((x,y,t)\in K\) 
		and show the lower Vietoris convergence
		of \(F^n(x,y,t)\to\SC\).
		
		Let \(V\subseteq \Q\) be an arbitrary open set
		with \(V\cap\SC\neq\emptyset\).
		Then \(Q(a,b;c,d)\cap\SC\neq\emptyset\)
		and \(Q(a,b;c,d)\subseteq V\)
		for some \(a<b\), \(c<d\).
		Hence, by Lemma \ref{lem:SC-holes}, 
		we have that \(U\cap\C\neq\emptyset\) 
		where \(U := (a,b)\times (c,d)\). 
		Since \(\C\) is the attractor 
		of the IFS \(\mathcal{G}\) 
		given by \eqref{eq:IFSforC}, \eqref{eq:mapsCarpet}, 
		there exists \(n_0\in\N\) such that
		\(U\cap G^n(x,y)\neq\emptyset\)
		for \(n\geq n_0\). Therefore 
		\begin{align*}
			V\cap F^n(x,y,t) \supseteq 
			U\times\{t\} \cap G^n(x,y)\times \{t\} 
			\neq \emptyset
		\end{align*}
		for \(n\geq n_0\),
		due to the relation 
		\(F^n(x,y,t)=G^n(x,y)\times\{t\}\).
	\end{proof}

	\section*{Acknowledgement}
	
	The first author would like 
	to thank Mateusz Maciejewski
	who dismissed one of the early ideas 
	that led to a trivial topology.

\end{document}